\documentclass[11pt]{article}
\usepackage[english]{babel}
\usepackage{a4wide}
\usepackage{amssymb, amsmath}
\usepackage{color}
\usepackage{enumitem}
\usepackage{amsthm}
\usepackage{wasysym}
\usepackage{stmaryrd}
\usepackage{tikz-cd}
\usetikzlibrary{arrows}
\tikzset{multimap/.tip={Glyph[glyph math command=multimap]},}
\usepackage{url}
\usepackage[capitalise]{cleveref}
\crefname{equation}{}{}

\newcommand{\ruitje}{\hfill $\lozenge$}
\newcommand{\pf}{\rightharpoonup}

\newcommand{\denotes}{\!\downarrow}
\newcommand{\s}{\hspace{1pt}}
\newcommand{\mono}{\hookrightarrow}

\newcommand{\am}{\multimap}
\newcommand{\un}{\textup{<}}
\newcommand{\adj}{\textup{adj}}

\DeclareMathOperator{\id}{id}
\DeclareMathOperator{\downset}{\downarrow}
\renewcommand{\sf}[1]{\mathsf{#1}}
\newcommand{\kbar}{\overline{\sf{k}}}
\newcommand{\asm}{\mathsf{Asm}}
\newcommand{\rt}{\mathsf{RT}}
\newcommand{\reg}{\mathsf{REG}}
\newcommand{\set}{\mathsf{Set}}

\newtheorem{thm}{Theorem}[section]
\newtheorem{lem}[thm]{Lemma}
\newtheorem{cor}[thm]{Corollary}
\newtheorem{prop}[thm]{Proposition}
\theoremstyle{definition}
\newtheorem{defn}[thm]{Definition}
\newtheorem{rem}[thm]{Remark}
\newtheorem{ex}[thm]{Example}

\title{On the Existence of Pushouts of Realizability Toposes}

\author{Jetze Zoethout\\
Department of Mathematics, Utrecht University}
\date{November 17, 2020}

\begin{document}

\maketitle

\centerline{\textsc{Abstract}}
{\footnotesize\setlength{\parindent}{0pt}
We consider two preorder-enriched categories of ordered PCAs: $\sf{OPCA}$, where the arrows are functional morphisms, and $\sf{PCA}$, where the arrows are applicative morphisms. We show that $\sf{OPCA}$ has small products and finite biproducts, and that $\sf{PCA}$ has finite coproducts, all in a suitable 2-categorical sense. On the other hand, $\sf{PCA}$ lacks all nontrivial binary products. We deduce from this that the pushout, over $\set$, of two nontrivial realizability toposes is never a realizability topos.}

\section{Introduction}

This paper is concerned with two categories of \emph{ordered} partial combinatory algebras (OPCAs). First, we study $\sf{OPCA}$, introduced by J.\@ van Oosten and P.\@ Hofstra in \cite{hofstrajaap}, where the arrows are \emph{functional morphisms}. Second, we consider the category $\sf{PCA}$, where the arrows are \emph{applicative morphisms}. Restricting the latter to discrete, i.e., unordered OPCAs yields the category of PCAs first introduced by J.\@ Longley in \cite{longleyphd}. Even though this category greatly facilitates the study of PCAs, not much is known about its categorical structure. Indeed, the comprehensive monograph \cite{jaap} (p.\@ 28) states: `It should be stressed that the category [of PCAs] is not very well understood at the moment of writing'. That moment was more than a decade ago, and since then, progress has been made (see, e.g., the paper \cite{faberjaap} by E.\@ Faber and J.\@ van Oosten). However, there is one construction available in this category that, to my knowledge, has thus far escaped attention or at least publication in the literature. It turns out that the category of PCAs has finite coproducts. Their construction, in the slightly more general setting of ordered PCAs, is described in the current paper.

A more general version of this construction already appeared in the paper \cite{Z19}, which discusses a category of generalized (but unordered) PCAs. The construction of coproducts in $\sf{PCA}$ below (\cref{sec:prod_PCA}) is a special case of this more general setting. One reason for presenting the construction here as well is to enable one to understand the construction of coproducts of OPCAs without having to work their way through the generalized PCAs from \cite{Z19}. Another reason is that, as we shall see below, coproducts of OPCAs interact in an interesting way with \emph{products} of OPCAs. In \cite{Z19}, the situation with products is quite different, and requires one to work over other `base categories' than the topos $\set$ of sets. In this paper, we will work exclusively over the base category $\set$. In the category of sets, we will freely assume the Axiom of Choice (AC); we will indicate the occasions where it is used.

The categories $\sf{OPCA}$ and $\sf{PCA}$ are enriched over preorders, so they carry a (simple) 2-categorical structure. Moreover, in the final section, we will briefly consider the 2-category of regular categories, and the 2-category of toposes, so some remarks on 2-categorical terminology are in order. In general, we will append the prefix `pseudo-' to a term to indicate that we define this term in a `fully weak' 2-categorical sense. Most importantly, a pseudolimit will be a limit where cones need only commute up to (specified, coherent) isomorphism, and whose universal property is expressed by an \emph{equivalence} of categories, rather than an isomorphism. Of course, in the preorder-enriched case, the isomorphisms need not be specified, since they are unique anyway. Observe that a pseudopullback officially specifies \emph{three} projecion morphisms, rather than two; but this will not play an important role in this paper. Pseudocolimits are defined completely analogously. It turns out that the pseudoproducts we construct below are actually \emph{2-products}, meaning that their universal property \emph{is} expressed by an isomorphism of categories. Note that we do not use the adjective `strict' here. We will use the adjective `strict' at another occasion, however: a strict pseudoinitial object will be a pseudoinitial object 0 with the additional property that every arrow $A\to 0$ is an equivalence. Similarly, we will use the term `strict pseudoterminal object' for the dual notion. Another important use of the prefix `pseudo-' concerns monos and epis. A 1-cell $f$ is called a pseudomono if postcomposition with $f$ is fully faithful. In the preorder-enriched case, this simply means that postcomposition with $f$ reflects the order. For epis, a similar definition applies.

The paper is structured as follows. First of all, in \cref{sec:OPCA}, we define the category $\sf{OPCA}$ and state some of its elementary properties. In \cref{sec:prod_OPCA}, we show that $\sf{OPCA}$ has small pseudoproducts (which are in fact 2-products) and finite pseudocoproducts, which also yield finite pseudobiproducts. Next, in \cref{sec:PCA}, we construct the category $\sf{PCA}$ from $\sf{OPCA}$. \cref{sec:prod_PCA} shows that the finite pseudocoproducts in $\sf{OPCA}$ also yield finite pseudocoproducts in $\sf{PCA}$. On the other hand, nontrivial binary pseudoproducts (i.e., where both factors are not the pseudoterminal object) never exist in $\sf{PCA}$. Finally, in \cref{sec:ass}, we deduce from this that the pushout, over $\set$, of two nontrivial realizability toposes is never itself a realizability topos.

\section{Ordered PCAs}\label{sec:OPCA}

In this section, we introduce ordered partial combinatory algebras and morphisms between them. Since we will not state any new results here, we will describe the important constructions, but omit most proofs.

A partial combinatory algebra is a nonempty set $A$ equipped with a \emph{partial} binary application map $(a,b)\mapsto ab$. We think of the elements of $A$ simultaneously as inputs and as (codes of) algorithms that act on these inputs. The element $ab$ stands for the output, if any, when the algorithm (with code) $a$ is applied to $b$. Of course, in order to capture the intuition that the application map is computation, this map will need to satisfy certain axioms, to be specified below.

A useful generalization of partial combinatory algebras was introduced by P.\@ Hofstra and J.\@ van Oosten \cite{hofstrajaap}. Here, a partial combinatory algebra $A$ is also equipped with a partial order $\leq$. We can think of the statement $a'\leq a$ as expressing that $a'$ gives more information than $a$, or that $a'$ is a specialization of $a$. Of course, this order will need to be compatible with the application map. Let us make this explicit.

\begin{defn}\label{defn:OPAS}
An \emph{ordered partial applicative structure} (OPAS) is a poset $A=(A,\leq)$ equipped with a partial binary map $A\times A\pf A, (a,b)\mapsto ab$ satisfying the following axiom:
\begin{itemize}
\item[(0)]	if $a'\leq a$, $b'\leq b$ and $ab$ is defined, then $a'b'$ is also defined, and $a'b'\leq ab$. \ruitje
\end{itemize}
\end{defn}
In other words, if $a'$ and $b'$ contain at least as much information as $a$ and $b$, and $ab$ is already defined, then $a'b'$ must also be defined and give at least as much information as $ab$.

Before we proceed to define ordered partial combinatory algebras, some remarks on notation are in order. First of all, the application map will not be associative, meaning that expressions involving application need to be bracketed properly. In order to prevent illegible expressions, we adopt the convention that application associates to the left, writing $abc$ as an abbreviation for $(ab)c$. Moreover, we will sometimes write $a\cdot b$ instead of $ab$ if this is necessary to avoid confusion.

Since the application map is partial, we also introduce some notation dealing with partiality. If $e$ is a possibly undefined expression, then we write $e\denotes$ to indicate that $e$ is in fact defined. We take this to imply that all subexpressions of $e$ are defined as well. If $e$ and $e'$ are two possibly undefined expressions, then we write $e'\preceq e$ for the statement: if $e\denotes$, then $e'\denotes$ and $e'\leq e$. On the other hand, $e'\leq e$ always expresses the stronger statement that $e'$ and $e$ are defined and satisfy $e'\leq e$. Observe that axiom (0) can also be written as: if $a'\leq a$ and $b'\leq b$, then $a'b'\preceq ab$. Moreover, we write $e\simeq e'$ if both $e'\preceq e$ and $e\preceq e'$. In other words, $e\simeq e'$ expresses the Kleene equality of $e$ and $e'$, meaning that $e\denotes$ iff $e'\denotes$, and in this case, $e$ and $e'$ denote the same value. On the other hand, $e=e'$ will always mean that $e$ and $e'$ are defined and equal to each other.

\begin{defn}\label{defn:OPCA}
An OPAS $A$ is an \emph{ordered partial combinatory algebra} (OPCA) if there exist $\sf{k},\sf{s}\in A$ satisfying:
\begin{itemize}
\item[(1)]	$\sf{k}ab \leq a$;
\item[(2)]	$\sf{s}ab\denotes$;
\item[(3)]	$\sf{s}abc\preceq ac(bc)$. \ruitje
\end{itemize}
\end{defn}
OPCAs satisfy an abstract version of the $Smn$ Theorem for Turing computability on the natural numbers. In order to make this precise, we need the following definition.
\begin{defn}
Let $A$ be an OPCA. The set of \emph{terms} over $A$ is defined recursively as follows:
\begin{itemize}
\item[(i)]	We assume given a countably infinite set of disinct variables, and these are all terms.
\item[(ii)]	For every $a\in A$, we assume that we have a \emph{constant symbol} for $a$, and this is a term. The constant symbol for $a$ is simply denoted by $a$.
\item[(iii)]	If $t_0$ and $t_1$ are terms, then so is $(t_0t_1)$. \ruitje
\end{itemize}
\end{defn}
We omit brackets whenever possible, again subject to the convention that application associates to the left. Moreover, we may write $t_0\cdot t_1$ if needed to avoid confusion.

Clearly, every \emph{closed} term $t$ can be assigned a (possibly undefined) interpretation in $A$, which will also be denoted by $t$. If $t(\vec{x})$ is a term in $n$ free variables, then this term defines an obvious partial function $A^n\pf A$, which sends a tuple $\vec{a}\in A^n$ to (the interpretation of) $t(\vec{a})$, if defined. The key fact about OPCAs is the all such functions are computable using an algorithm present in $A$.
\begin{prop}[Combinatory completeness]\label{prop:combcomp}
Let $A$ be an OPCA. There exists a map that assigns, to each term $t(\vec{x},y)$ in $n+1$ variables, an element $\lambda^\ast \vec{x}y. t$ of $A$, satisfying:
\begin{itemize}
\item	$(\lambda^\ast \vec{x}y. t)\vec{a}\denotes$;
\item	$(\lambda^\ast \vec{x}y. t)\vec{a}b\preceq t(\vec{a},b)$,
\end{itemize}
for all $\vec{a}\in A^n, b\in A$.
\end{prop}
The proof is an easy adaptation of the proof of Theorem 1.1.3 in \cite{jaap}, and is omitted. It is worth mentioning that the map $t(\vec{x},y) \mapsto \lambda^\ast \vec{x}y.t$ can be constructed explicitly and only requires a choice for $\sf{k}$ and $\sf{s}$ as in \cref{defn:OPCA}.

The elements $\sf{k}$ and $\sf{s}$ are usually called \emph{combinators}. Using $\sf{k}$, $\sf{s}$ and \cref{prop:combcomp}, we can construct additional useful combinators. For our purposes, the combinators $\sf{i} = \sf{skk}$, $\kbar = \sf{ki}$, $\sf{p} = \lambda^* xyz.zxy$, $\sf{p}_0 = \lambda^* x.x\sf{k}$ and $\sf{p}_1 = \lambda^* x.x\kbar$ will be relevant. These satisfy:
\[
\sf{i}a\leq a, \quad \kbar ab\leq b, \quad \sf{p}_0(\sf{p}ab)\leq a \quad\mbox{and}\quad \sf{p}_1(\sf{p}ab)\leq b.
\]
The combinators $\sf{k}$ and $\kbar$ also serve as \emph{booleans}, meaning that there exists a case combinator $\sf{C}\in A$ satisfying $\sf{Ck}ab\leq a$ and $\sf{C}\kbar ab\leq b$. Observe that we may simply take $\sf{C}=\sf{i}$.

\begin{rem}
Even though $\sf{k}$ and $\sf{s}$ are not part of the structure of an OPCA, we will assume that, for each OPCA we discuss, we have made an explicit choice for $\sf{k}$ and $\sf{s}$. Observe that this also yields a choice for the other combinators constructed above. If one has a lot of OPCAs, then this may require the Axiom of Choice; this situation will occur in the proof of \cref{prop:prod_OPCA}. \ruitje
\end{rem}

\begin{ex}
The prototypical example is the (discretely ordered) OPCA $\mathcal{K}_1$, known as Kleene's first model. Its underlying set is the set of natural numbers, and $mn$ is the result, if any, when the $m$-th partial recursive function is applied to $n$. \ruitje
\end{ex}

\begin{ex}
Any poset with binary meets is an OPCA, where application is given by meet. These are examples of \emph{pseudotrivial} OPCAs (\cite{hofstrajaap}, Definition 2.3), i.e., OPCAs where any two elements have a common lower bound. This notion will not play a large role in this paper; we will need it only in \cref{ex:cd_prod_OPCA} below. \ruitje
\end{ex}

We now proceed to define maps between OPCAs.
\begin{defn}\label{defn:morph}
Let $A$ and $B$ be OPCAs. A \emph{morphism of OPCAs} is a function $f\colon A\to B$ satisfying the following requirements:
\begin{itemize}
\item	there exists a $t\in B$ such that $t\cdot f(a)\cdot f(a') \preceq f(aa')$;
\item	there exists a $u\in B$ such that $u\cdot f(a')\leq f(a)$ whenever $a'\leq a$.
\end{itemize}
We say that $t$ \emph{tracks} $f$ and that $f$ preserves the order up to $u$. \ruitje
\end{defn}
\begin{defn}\label{defn:order_morph}
Let $A$ and $B$ be OPCAs and consider two functions $f,f'\colon A\to B$. We say that $f\leq f'$ if there exists an $s\in B$ such that $s\cdot f(a)\leq f'(a)$ for all $a\in A$. Such an $s\in B$ is said to \emph{realize} the inequality $f\leq f'$. Moreover, we write $f\simeq f'$ if both $f\leq f'$ and $f'\leq f$. \ruitje
\end{defn}

\begin{prop}\label{prop:OPCA}
OPCAs, morphisms of OPCAs and inequalities between them form a preorder-enriched category $\sf{OPCA}$.
\end{prop}

We will be espacially interested in morphisms with the following property, introduced in \cite{hofstrajaap}.
\begin{defn}\label{defn:cd}
Let $f\colon A\to B$ be a morphism of OPCAs. We say that $f$ is \emph{computationally dense} (c.d.) if there exists an $n\in B$ satisfying:
\begin{align}\tag{cd}\label{eq:cd}
\forall s\in B\s \exists r\in A\s (n\cdot f(r)\leq s).
\end{align}
\ \hfill $\lozenge$
\end{defn}

In \cref{sec:prod_PCA}, we will also need the following notion.
\begin{defn}
A morphism of OPCAs $f\colon A\to B$ is called \emph{discrete} if, for any subset $X\subseteq A$, we have: if $f(X) = \{f(a)\mid a\in X\}$ has a lower bound in $B$, then $X$ has a lower bound in $A$. \ruitje
\end{defn}

We list some elementary properties of computational density and discreteness, which we leave to the reader to prove.
\begin{prop}\label{prop:cd}
Let $A\stackrel{f}{\longrightarrow} B\stackrel{g}{\longrightarrow} C$ be morphisms of OPCAs.
\begin{itemize}
\item[\textup{(}i\textup{)}]	If $f$ and $g$ are c.d., then $gf$ is c.d.\@ as well.
\item[\textup{(}ii\textup{)}]	If $gf$ is c.d., then $g$ is c.d.\@ as well.
\item[\textup{(}iii\textup{)}]	If $gf$ is discrete, then $f$ is discrete as well.
\item[\textup{(}iv\textup{)}]	Computational density and discreteness are downwards closed. That is, if $f$ is c.d.\@ (resp.\@ discrete) and $f'\leq f$ is a morphism of OPCAs, then $f'$ is also c.d.\@ (resp.\@ discrete).
\end{itemize}
In particular, left adjoints are c.d., and right adjoints are discrete.
\end{prop}

The definition of computational density in \cref{defn:cd} is not the original definition from \cite{hofstrajaap}, but rather a simplified version introduced by P.\@ Johnstone. The following proposition provides the original definition from \cite{hofstrajaap}, which we will need later on.
\begin{prop}[\cite{johnstone}, Lemma 2.3]\label{prop:cdm}
A morphism of OPCAs $f\colon A\to B$ is c.d.\@ if and only if there exists an $m\in B$ satisfying:
\begin{align}\tag{cdm}\label{eq:cdm}
\forall s\in B\s \exists r\in A\s\forall a\in A\ (m\cdot f(ra)\preceq s\cdot f(a)).
\end{align}
In fact, any $m\in B$ satisfying \cref{eq:cdm} also satisfies \cref{eq:cd}.
\end{prop}
\begin{proof}
First of all, suppose that $m\in B$ satisfies \cref{eq:cdm}. If $s\in B$, then we know that $\sf{k}s$ is defined, so by \cref{eq:cdm}, there exists an $r\in A$ such that $m\cdot f(ra)\preceq \sf{k}s\cdot f(a) \leq s$ for all $a\in A$. In particular, we have $m\cdot f(r\sf{i})\leq s$, so $m$ satisfies \cref{eq:cd}.

Conversely, suppose that $n\in B$ satisfies \cref{eq:cd}. Let $t\in B$ we a tracker of $f$ and let $f$ preserve the order up to $u\in B$. We define
\[
m = \lambda^* x. n(u(t\cdot f(\sf{p}_0)\cdot x))(u(t\cdot f(\sf{p}_1)\cdot x)).
\]
Now let $s\in B$, and find an $r\in A$ such that $n\cdot f(r)\leq s$. Now we compute
\begin{align*}
m\cdot f(\sf{p}ra) &\preceq n(u(t\cdot f(\sf{p}_0)\cdot f(\sf{p}ra)))(u(t\cdot f(\sf{p}_1)\cdot f(\sf{p}ra)))\\
&\preceq n(u \cdot f(\sf{p}_0(\sf{p}ra)))(u \cdot f(\sf{p}_1(\sf{p}ra)))\\
&\preceq n\cdot f(r)\cdot f(a)\\
&\preceq s\cdot f(a),
\end{align*}
as desired.
\end{proof}


\section{Products and coproducts in $\mathsf{OPCA}$}\label{sec:prod_OPCA}

In this section, we investigate the existence of pseudo(co)products in $\sf{OPCA}$, and their interaction with c.d.\@ morphisms. We start with a result by J.\@ Longley (\cite{longleyphd}, Proposition 2.1.7).

\begin{prop}\label{prop:zero_obj}
The category $\sf{OPCA}$ has a pseudozero object.
\end{prop}
\begin{proof}
The required pseudozero object is the OPCA $\mathbf{1} = \{*\}$, where $** = *$. For every OPCA $A$, there is only one function $!\colon A\to \mathbf{1}$, and this is clearly a morphism of OPCAs, so $\mathbf{1}$ is in fact a 2-terminal object. Conversely, every element $c\in A$ yields a morphism of OPCAs $\un\colon 1\to A$ with $\un(*) = c$. Clearly, these are all isomorphic, so $\mathbf{1}$ is also a pseudoinitial object.
\end{proof}

The existence of a pseudozero object means that we also have \emph{zero morphisms}.

\begin{defn}\label{defn:zero_mor}
A morphism of OCPAs $A\to B$ is called a \emph{zero morphism} if it factors, up to isomorphism, through $\mathbf{1}$. 
\end{defn}
The following lemma provides two alternative characterizations of zero morphisms. We leave the proof to the reader.
\begin{lem}\label{lem:zero_mor}
For a morphism of OPCAs $f\colon A\to B$, the following are equivalent:
\begin{itemize}
\item[\textup{(}i\textup{)}]	$f$ is a zero morphism;
\item[\textup{(}ii\textup{)}]	$f(A) = \{f(a)\mid a\in A\}$ has a lower bound;
\item[\textup{(}iii\textup{)}]	$f$ is a top element of $\sf{OPCA}(A,B)$.
\end{itemize}
\end{lem}
It follows from (iii) that $\sf{OPCA}$ is even enriched over preorders with a top element. Before we continue, we characterize the OPCA $\mathbf{1}$ up to equivalence in a number of ways.
\begin{lem}\label{lem:trivial}
Let $A$ be an OPCA. The following are equivalent:
\begin{itemize}
\item[\textup{(}i\textup{)}]	$A$ is equivalent to $\mathbf{1}$;
\item[\textup{(}ii\textup{)}]	$A$ has a least element;
\item[\textup{(}iii\textup{)}]	$\id_A$ is a zero morphism;
\item[\textup{(}iv\textup{)}]	$\un\colon \mathbf{1}\to A$ is c.d.
\end{itemize}
\end{lem}
An OPCA $A$ satisfying the equivalent conditions of \cref{lem:trivial} will be called \emph{trivial}.

If $A$ is an OPCA, then $!\circ\un$ is isomorphic to the identity $\id_\mathbf{1}$. On the other hand, $\un\s\circ \s !$ is, by definition, a zero morpism, so we also have $\id_A \leq \un\s\circ\s !$. This means that $!\dashv \un$.

In \cite{hofstrajaap} (Remark (2) on p.\@ 450), it is observed that $\sf{OPCA}$ has binary products. This construction generalizes to products of arbitrary (small) size, given choice on the index set.
\begin{prop}\label{prop:prod_OPCA}
The category $\sf{OPCA}$ has small pseudoproducts.
\end{prop}
\begin{proof}
Suppose we have an $I$-indexed sequence of OPCAs $(A_i)_{i\in I}$. We equip the product $A = \prod_{i\in I} A_i$ with an OPAS structure by defining the order and application coordinatewise. That is, if $a = (a_i)_{i\in I}$ and $b=(b_i)_{i\in I}$ are elements of $A$, then we set
\begin{itemize}
\item	$a\leq b$ iff $a_i\leq b_i$ for all $i\in I$;
\item	$ab\denotes$ iff $a_ib_i\denotes$ for all $i\in I$, and in this case, $ab = (a_ib_i)_{i\in I}$.
\end{itemize}
Observe that $A$ is nonempty by AC, and axiom (0) clearly holds for $A$, since it holds coordinatewise. For all $i\in I$, we may (using AC) pick suitable combinators $\sf{k}_i$ and $\sf{s}_i$ for $A_i$. Then it is not hard to check that $\sf{k} = (\sf{k}_i)_{i\in I}$ and $\sf{s} = (\sf{s}_i)_{i\in I}$ are suitable combinators for $A$, so $A$ is an OPCA. Moreover, for each $i\in I$, the projection $\pi_i\colon A\to A_i$ is easily seen to be a morphism of OPCAs.

Now suppose we have an OPCA $B$ and for all $i\in I$, a morphism $f_i\colon B\to A_i$. Then we have the obvious amalgamation $f = \langle f_i\rangle_{i\in I}\colon b\mapsto (f_i(b))_{i\in I}$. If, for each $i\in I$, we pick a tracker $t_i\in A_i$ of $f_i$, then $t = (t_i)_{i\in I}$ tracks $f$. Similarly, $f$ preserves the order up to $u = (u_i)_{i\in I}$, where each $f_i$ preserves the order up to $u_i\in A_i$. This shows that $f$ is a morphism of OPCAs, and we clearly have $\pi_if=f_i$ for all $i\in I$.

Finally, suppose we have $g,g'\colon B\to A$ such that $\pi_ig\leq \pi_ig'$ for all $i\in I$. If we pick, for each $i\in I$, a realizer $s_i\in A_i$ of $\pi_ig\leq \pi_ig'$, then $s = (s_i)_{i\in I}$ realizes $g\leq g'$. This concludes the proof, and we see that $\prod_{i\in I} A_i$ is even the 2-product of the $A_i$.
\end{proof}

The projections $\pi_i$ are clearly c.d., so if an amalgamation $f = \langle f_i\rangle_{i\in I}$ is c.d., then so are all the $f_i$. The converse only holds for \emph{finite} products.
\begin{prop}\label{prop:cd_prod_OPCA}
If $(A_i)_{i\in I}$ is a finite sequence of OPCAs, and the morphisms $f_i\colon B\to A_i$ are c.d., then $\langle f_i\rangle_{i\in I}\colon B\to\prod_{i\in I} A_i$ is also c.d.
\end{prop}
\begin{proof}
It suffices to treat the nullary and the binary case. The nullary case states that $!\colon B\to \mathbf{1}$ is always c.d., which follows from the adjunction $!\dashv\un$.

For the binary case, suppose we have c.d.\@ morphisms $f_0\colon B\to A_0$ and $f_1\colon B\to A_1$. Let $t_i\in A_i$ track $f_i$, let $f_i$ preserve the order up to $u_i\in A_i$, and let the computational density of $f_i$ be witnessed by $n_i\in A_i$. We define $n'_i = \lambda^* x.n_i(u_i(t_i\cdot f_i(\sf{p}_i)\cdot x))\in A_i$. We claim that $n=(n'_0,n'_1)\in A_0\times A_1$ witnesses the computational density of $f=\langle f_0,f_1\rangle\colon B\to A_0\times A_1$.

In order to prove this, let $s=(s_0,s_1)\in A_0\times A_1$. Then we know that there exist $r_i\in B$ such that $n_i\cdot f_i(r_i)\leq s_i$. Now define $r = \sf{p}r_0r_1\in B$. Then
\[
n'_i\cdot f_i(r) \preceq n_i(u_i(t_i\cdot f_i(\sf{p}_i)\cdot f_i(r))) \preceq n_i(u_i\cdot f_i(\sf{p}_ir)) \preceq n_i\cdot f(r_i) \leq s_i,
\]
so $n\cdot f(r)\leq s$, as desired.
\end{proof}

\begin{ex}\label{ex:cd_prod_OPCA}
Let $A$ be an OPCA that is not pseudotrivial. Then in particular, $\sf{k}$ and $\kbar$ do not have a common lower bound, for if $u$ were a lower bound of $\sf{k}$ and $\kbar$, then $uab$ would be a lower bound of $a$ and $b$, for arbitrary $a,b\in A$. Let $I$ be a set such that $2^{|I|}>|A|$. Then a morphism $f\colon A\to A^I$ is never c.d., where $A^I$ denotes the $I$-fold product of $A$. Indeed, suppose for the sake of contradiction that $f$ is c.d., witnessed by $n\in A^I$. Then every element of $A^I$ is bounded from below by an element of $X = \{n\cdot f(r)\mid r\in A, n\cdot f(r)\denotes\}$. This set $X$ has cardinality at most $|A|$. However, the subset $\{a\in A^I\mid \forall i\in I\s (a_i\in\{\sf{k},\kbar\})\}$ of $A^I$, which has cardinality $2^{|I|}>|A|\geq |X|$, has the property that every two distinct elements do not have a common lower bound in $A^I$: contradiction.

In particular, the diagonal $\delta\colon A\to A^I$ is not c.d., which means that \cref{prop:cd_prod_OPCA} does not hold for infinite $I$. \ruitje
\end{ex}

Just as the 2-terminal object $\mathbf{1}$ is also pseudoinitial, \emph{finite} 2-products in $\sf{OPCA}$ also serve as pseudocoproducts.
\begin{thm}
The category $\sf{OPCA}$ has finite pseudocoproducts.
\end{thm}
\begin{proof}
It suffices to treat the binary case. Let $A_0$ and $A_1$ be OPCAs. Then there is a morphism of OPCAs $\kappa_0\colon A_0\to A_0\times A_1$ given by $\kappa_A(a) = (a,\sf{i})$. Similarly, we have $\kappa_1\colon A_1\to A_0\times A_1$ given by $\kappa_1(a) = (\sf{i},a)$. We claim that this is a pseudocoproduct diagram.

First of all, suppose that we have morphisms of OPCAs $f_0\colon A_0\to B$ and $f_1\colon A_1\to B$. Let $t_i\in B$ track $f_i$, and let $f_i$ preserve the order up to $u_i\in B$. We define $f=[f_0,f_1]\colon A_0\times A_1\to B$ by $f(a_0,a_1) = \sf{p}\cdot f_0(a_0)\cdot f_1(a_1)$. Then $f$ is tracked by
\[
\lambda^* xy. \sf{p}(t_0(\sf{p}_0x)(\sf{p}_0y))(t_1(\sf{p}_1x)(\sf{p}_1y))\in B,
\]
as a straightforward calculation will show. Similarly, one can show that $f$ preserves the order up to $\lambda^* x.\sf{p}(u_0(\sf{p}_0x))(u_1(\sf{p}_1x))\in B$, so $f$ is a morphism of OPCAs. We have $f(\kappa_0(a)) = \sf{p}a\sf{i}$, so $\sf{p}_0\in B$ realizes $f\kappa_0\leq f_0$ and $\lambda^* x. \sf{p}x\sf{i}$ realizes $f_0\leq f\kappa_0$. Similarly, one shows that $f\kappa_1\simeq f_1$.

Now suppose we have morphisms $g,g'\colon A_0\times A_1\to B$ such that $g\kappa_0\leq g'\kappa_0$ and $g\kappa_1\leq g'\kappa_1$. Let $s_i\in B$ realize $g\kappa_i\leq g'\kappa_i$, let $t,t'\in B$ track $g$ resp.\@ $g'$, and suppose that $g$ and $g'$ preserve the order up to $u,u'\in B$ respectively. We claim that $g\leq g'$ is realized by:
\[
s = \lambda^* x.u'(t'(t'\cdot g'(\sf{k},\kbar)\cdot(s_0(u(t\cdot g(\sf{i},\sf{ki})\cdot x))))(s_1(u(t\cdot g(\sf{ki},\sf{i})\cdot x)))) \in B.
\]
Let $(a_0,a_1)\in A_0\times A_1$. Then we have:
\begin{align*}
s_0(u(t\cdot g(\sf{i},\sf{ki})\cdot g(a_0,a_1))) &\preceq s_0(u\cdot g(\sf{i}a_0,\sf{ki}a_1))\\
&\preceq s_0\cdot g(a_0,\sf{i})\\
&\simeq s_0\cdot g(\kappa_0(a_0))\\
&\leq g'(\kappa_0(a_0))\\
&= g'(a_0,\sf{i}),
\end{align*}
and similarly, $s_1(u(t\cdot g(\sf{ki},\sf{i})\cdot g(a_0,a_1)))\leq g'(\sf{i},a_1)$. This yields:
\begin{align*}
s\cdot g(a_0,a_1) &\preceq u'(t'(t'\cdot g'(\sf{k},\kbar)\cdot g'(a_0,\sf{i}))\cdot g'(\sf{i},a_1))\\
&\preceq u'(t'\cdot g'(\sf{k}a_0,\kbar\sf{i})\cdot g'(\sf{i},a_1))\\
&\preceq u'\cdot g(\sf{k}a_0\sf{i}, \kbar\sf{i}a_1)\\
&\leq g'(a_0,a_1), 
\end{align*}
as desired.
\end{proof}

\begin{cor}
The category $\sf{OPCA}$ has finite pseudobiproducts
\end{cor}
\begin{proof}
The only thing left to check is that $A_0\stackrel{\kappa_0}{\longrightarrow} A_0\times A_1\stackrel{\pi_0}{\longrightarrow} A_0$ is isomorphic to $\id_{A_0}$, and that $A_0\stackrel{\kappa_0}{\longrightarrow} A_0\times A_1\stackrel{\pi_1}{\longrightarrow} A_1$ is a zero morphism. Both are immediate.
\end{proof}

Moreover, \cref{prop:cd}(ii) immediately yields the following relation between coproducts and computational density.

\begin{cor}
If $f_0\colon A_0\to B$ and $f_1\colon A_1\to B$ are morphisms of OPCAs and $f_0$ is c.d., then $[f_0,f_1]\colon A_0\times A_1\to B$ is also c.d.
\end{cor}

In analogy with ordinary coproducts, we say that finite pseudocoproducts are \emph{disjoint} if, for every pseudocoproduct diagram $A_0\to A_0\sqcup A_1\leftarrow A_1$, the coprojections are pseudomonos, and
\[\begin{tikzcd}
0 \arrow[d] \arrow[r] \arrow[rd] & A_1 \arrow[d] \\
A_0 \arrow[r]                    & A_0\sqcup A_1
\end{tikzcd}\]
is a pseudopullback, where 0 denotes the pseudoinitial object.

\begin{prop}\label{prop:OPCA_disjoint}
The finite pseudocoproducts in $\sf{OPCA}$ are disjoint.
\end{prop}
\begin{proof}
Since $\pi_i\kappa_i\simeq \id_{A_i}$, it is immediate that the $\kappa_i$ are pseudomonos. In order to establish the required pseudopullback, we need to show the following: if we have morphisms $f_0\colon B\to A_0$ and $f_1\colon B\to A_1$ such that $\kappa_0f_0\simeq \kappa_1f_1$, then $f_0$ and $f_1$ are both zero morphisms. Let $s = (s_0,s_1) \in A_0\times A_1$ realize $\kappa_0f_0 \leq \kappa_1f_1$. Then for all $b\in B$, we have $(s_0\cdot f_0(b),s_1\sf{i}) \simeq s\cdot \kappa_0(f_0(b)) \leq \kappa_1(f_1(b)) = (\sf{i},f_1(b))$. In particular, we have $s_1\sf{i}\leq f_1(b)$ for all $b\in B$, so $f_1$ is a zero morphism. The proof that $f_0$ is a zero morphism proceeds analogously.
\end{proof}
The `dual' result to \cref{prop:OPCA_disjoint} also holds; this will be useful in \cref{sec:prod_PCA}.
\begin{prop}\label{prop:dual_disjoint}
If $A_0$ and $A_1$ are OPCAs, then $\pi_i\colon A_0\times A_1\to A_i$ is a pseudoepi and
\[\begin{tikzcd}
A_0\times A_1 \arrow[d] \arrow[r] \arrow[rd] & A_1 \arrow[d] \\
A_0 \arrow[r]                                & \mathbf{1}   
\end{tikzcd}\]
is a pseudopushout diagram.
\end{prop}
\begin{proof}
Since $\pi_i\kappa_i\simeq \id_{A_i}$, we know that $\pi_i$ is indeed pseudoepi.

For the pseudopushout, we need to show the following: if $f_0\colon A_0\to B$ and $f_1\colon A_1\to B$ are morphisms such that $f_0\pi_0\simeq f_1\pi_1$, then $f_0$ and $g_0$ are both zero morphisms. If $s\in B$ realizes $f_0\pi_0\leq f_1\pi_1$, then we have $s\cdot f_0(a_0)\leq f_1(a_1)$ for all $a_0\in A_0$ and $a_1\in A_1$. In particular, we have $s\cdot f_0(\sf{i})\leq f_1(a_1)$ for all $a_1\in A_1$, so $f_1$ is a zero morphism. The proof that $f_0$ is a zero morphism again proceeds analogously.
\end{proof}

We close this section by investigating coproducts in a category related to $\sf{OPCA}$.

\begin{defn}\label{defn:OPCAadj}
The preorder-enriched category $\sf{OPCA}_\adj$ is defined as follows.
\begin{itemize}
\item	Its objects are OPCAs.
\item	An arrow $f\colon A\to B$ is a pair of morphisms $f^*\colon B\to A$ and $f_*\colon A\to B$ with $f^*\dashv f_*$.
\item	If $f,g\colon A\to B$, then we say that $f\leq g$ if $f^*\leq g^*$; equivalently, if $g_*\leq f_*$. \ruitje
\end{itemize}
\end{defn}

\begin{prop}\label{prop:coprod_OPCA_adj}
The category $\sf{OPCA}_\adj$ has finite pseudocoproducts. Moreover, the pseudoinitial object is strict, and pseudocoproducts are disjoint.
\end{prop}
\begin{proof}
We have already seen that there are essentially unique morphisms $!\colon A\to \mathbf{1}$ and $\un\colon \mathbf{1}\to A$ satisfying $!\dashv\un$, yielding the (essentially) unique arrow $\mathbf{1}\to A$ in $\sf{OPCA}_\adj$. Moreover, if we have an arrow $A\to \mathbf{1}$ in $\sf{OPCA}_\adj$, then also $\un\!\dashv\ !$, so $!$ and $\un$ form an equivalence between $A$ and $\mathbf{1}$, meaning that $\mathbf{1}$ is indeed strict.

Now consider two OPCAs $A$ and $B$. We have the product diagram $A\stackrel{\pi_A}{\longleftarrow} A\times B\stackrel{\pi_B}{\longrightarrow} B$ and the coproduct diagram $A\stackrel{\kappa_A}{\longrightarrow} A\times B\stackrel{\kappa_B}{\longleftarrow} B$. We have already remarked that $\pi_A\kappa_A\simeq\id_A$. Moreover, it is easily computed that $\kappa_A\pi_A\geq\id_{A\times B}$, which means that $\pi_A\dashv\kappa_A$ is an arrow $A\to A\times B$ of $\sf{OPCA}_\adj$. Similarly, we have the arrow $\pi_B\dashv\kappa_B\colon B\to A\times B$. In order to show that this yields a pseudocoproduct diagram in $\sf{OPCA}_\adj$, we need to show the following: if $f\colon A\to C$ and $g\colon B\to C$ are arrows of $\sf{OPCA}_\adj$, then $h^* = \langle f^*, g^*\rangle$ is left adjoint to $h_* = [f_*,g_*]$. First of all, we may easily compute that $h_*(h^*(c)) = \sf{p}\cdot f_*(f^*(c))\cdot g_*(g^*(c))$. So, if $r,s\in C$ realize $\id_C\leq f_* f^*$ and $\id_C\leq g_* g^*$ respectively, then $\lambda^* x. \sf{p}(rx)(sx)$ realizes $\id_C\leq h_* h^*$. The other inequality can be obtained completely from universal properties. We have:
\[
\pi_A h^* h_* \kappa_A \simeq f^* f_* \leq \id_A \simeq \pi_A\kappa_A\quad\mbox{and}\quad \pi_A h^* h_* \kappa_B \simeq f^* g_* \leq \pi_A\kappa_B,
\]
so from the universal property of the coproduct $A\times B$, it follows that $\pi_A h^* h_* \leq\pi_A$. Similarly, we obtain $\pi_B h^* h_*\leq \pi_B$, and the universal property of the product $A\times B$ yields $h^* h_*\leq \id_{A\times B}$, as desired.

For disjointness, we first note that $\pi_A\dashv\kappa_A$ is a pseudomono because $\pi_A\kappa_A\simeq \id_A$. Now suppose we have arrows $f\colon C\to A$ and $g\colon C\to B$ of $\sf{OPCA}_\adj$ such that $\kappa_A f_*\simeq \kappa_B g_*$. Then we know from \cref{prop:OPCA_disjoint} that $f_*$ and $g_*$ are both zero morphisms. From $\id_C\geq f^* f_*$, it follows that $\id_C$ is also a zero morphism, i.e., $C$ is trivial. Now it is immediate that $\mathbf{1}$ is the pseudopullback of $A\to A\times B\leftarrow B$ in $\sf{OPCA}_\adj$.
\end{proof}

In particular, we must have that the codiagonal $\varepsilon\colon A\times A\to A$ is right adjoint to the diagonal $\delta\colon A\to A\times A$. This means that we can view $\varepsilon$ as an `internal binary meet map' on $A$ (compare with the internal finite meets of BCOs in \cite{allrealisrel}, p.\@ 246). Explicitly, this map is given by $\varepsilon(a,a') = \sf{p}aa'$. We can also deduce from this that $\sf{OPCA}$ is even enriched over posets with finite meets, rather than posets with a top element.


\section{Applicative morphisms}\label{sec:PCA}

In this section, we introduce the category of ordered PCAs and \emph{applicative} morphisms between them. Applicative morphisms (between unordered PCAs) were the morphisms originally considered by J.\@ Longley in \cite{longleyphd}. Applicative morphisms are no longer functions between the underlying sets, but total relations. In \cite{hofstrajaap}, it is shown how to reconstruct the notion of applicative morphism by introducting a certain pseudomonad on $\sf{OPCA}$. This is also the treatment we follow here.

\begin{defn}
Let $A$ be an OPCA.
\begin{itemize}
\item[(i)]	We define a new OPCA $TA$ as follows:
\begin{itemize}
\item	$TA$ is the set of all nonempty downsets of $A$, i.e.,
\[
TA = \{\emptyset\neq\alpha\subseteq A\mid \mbox{if }a\in\alpha\mbox{ and }a'\leq a,\mbox{ then }a'\in\alpha\}.
\]
\item	$TA$ is ordered by inclusion.
\item	For $\alpha,\beta\in TA$, we say that $\alpha\beta\denotes$ iff $ab\denotes$ for all $a\in\alpha$ and $b\in\beta$; and in this case,
\[
\alpha\beta = \downset\{ab\mid a\in\alpha, b\in\beta\}.
\]
\end{itemize}
\item[(ii)]	For a morphism of OPCAs $f\colon A\to B$, we define $Tf\colon TA\to TB$ by $Tf(\alpha) = \downset f(\alpha) = \downset\{f(a)\mid a\in \alpha\}$.
\item[(iii)]	We define $\delta_A\colon A\to TA$ and $\bigcup_A\colon TTA\to TA$ by $\delta_A(a) = \downset\{a\}$ and $\bigcup_A(\mathcal{A}) = \bigcup\mathcal{A}$. \ruitje
\end{itemize}
\end{defn}
Observe that for the combinators in $TA$, we may simply take $\downset\{\sf{k}\}$ and $\downset\{\sf{s}\}$.
\begin{prop}
The triple $(T,\delta,\bigcup)$ is a KZ-pseudomonad on $\sf{OPCA}$.
\end{prop}
The proof is very similar to case of the nonempty downset monad on the category of posets, but one has to insert some realizers at appropriate positions. We leave this to the reader.

\begin{defn}
The preorder-enriched category $\sf{PCA}$ is defined as the Kleisli category for the pseudomonad $T$. An arrow of $\sf{PCA}$ will be called an \emph{applicative morphism}, and will be denoted by $f\colon A\am B$. \ruitje
\end{defn}

Let us consider for a moment what this means. The objects of $\sf{PCA}$ are still OPCAs. An applicative morphism $f\colon A\am B$ is a morphism of OPCAs $f\colon A\to TB$. This means that $f$ does not assign an \emph{element} of $B$ to $a\in A$, but rather a (nonempty and downwards closed) set of elements. For this reason, we use the multimap sign $\am$ for applicative morphisms. The identity on $A$ is $\delta_A$, and the composition of $f\colon A\am B$ and $g\colon B\am C$ is $\bigcup_C\circ Tg\circ f$, i.e., $gf(c) = \bigcup_{b\in f(a)} g(b)$. The requirements for an applicative morphism can be reformulated completely in terms of elements of $B$ (rather than $TB$). It is convenient to use the following notation: if $a\in A$ and $\alpha\in TA$, then we write
\[
a\cdot\alpha := \downset\{a\}\cdot\alpha = \downset\{aa'\mid a'\in\alpha\}.
\]
Now, a function $f\colon A\to TB$ is an applicative morphism iff the following hold:
\begin{itemize}
\item	There exists an $r\in B$ such that $r\cdot f(a)\cdot f(a')\subseteq f(aa')$ whenever $aa'\denotes$; such an $r$ will also be called a tracker of $f$ (even though the tracker is really $\downset\{r\}\in TB$).
\item	There exists a $u\in B$ such that $u\cdot f(a')\subseteq f(a)$ whenever $a'\leq a$. We will say that $f$ preserves the order up to $u$
\end{itemize}
Similarly, if $f,f'\colon A\am B$, then we have that $f\leq f'$ iff there exists an $s\in B$ such that $s\cdot f(a)\subseteq f'(a)$ for all $a\in A$; and such an $s$ will be called a realizer of $f\leq f'$.

It turns out for applicative morphisms, one can get rid of the realizer $u$ above.
\begin{lem}
Every applicative morphism is isomorphic to an order-preserving applicative morphism.
\end{lem}
\begin{proof}
Given $f\colon A\am B$, define $f'\colon A\am B$ by $f'(a) = \bigcup_{a'\leq a} f(a')$. Clearly, $\sf{i}\in B$ realizes $f\leq f'$, and if $f$ preserves the order up to $u\in B$, then $u$ realizes $f'\leq f$. So we have $f\simeq f'$, which also implies that $f'$ is, in fact, an applicative morphism. Clearly, $f'$ preserves the order on the nose.
\end{proof}

If $f\colon A\am B$ is an applicative morphism, then there exists an essentially unique \emph{$T$-algebra morphism} $\tilde{f}\colon TA\to TB$ such that the diagram
\[\begin{tikzcd}
A \arrow[rr, "f"] \arrow[rd, "\delta_A"'] &                             & TB \\
                                          & TA \arrow[ru, "\tilde{f}"'] &   
\end{tikzcd}\]
commutes. Explicitly, we have $\tilde{f} \simeq \bigcup_B\circ Tf$. It is well known from the general theory of (pseudo)monads that this yields an equivalence between $\sf{PCA}$ and the full subcategory of $T$-$\sf{Alg}$ on the free $T$-algebras. Moreover, it is easy to show that $\delta_A$ is c.d., so \cref{prop:cd} implies that $f$ is c.d.\@ iff $\tilde{f}$ is c.d. This means we have an unambiguous notion of computational density for applicative morphisms. Explicitly, there should be an $n\in B$ such that
\[
\forall s\in B\s \exists r\in A\s (n\cdot f(r)\subseteq\downset \{s\}).
\]
The results from \cref{prop:cd} automatically hold for $\sf{PCA}$ as well. For example, suppose that $f\colon A\am B$ and $g\colon B\am C$ are c.d. Then $\tilde{f}$ and $\tilde{g}$ are c.d., so by \cref{prop:cd}(i), $\widetilde{gf} \simeq \tilde{g}\tilde{f}$ is c.d., hence $gf$ is c.d.

Moreover, there exists a pseudofunctor $\sf{OPCA}\to\sf{PCA}$ sending a morphism $f\colon A\to B$ to $\delta_Bf\colon A\am B$. Because $\delta_B$ is always a pseudomono, this pseudofunctor is fully faithful on 2-cells. Furthermore, one easily shows that this pseudofunctor preserves and reflects computational density.
\begin{defn}
An applicative morphism $f\colon A\am B$ is called \emph{projective} if $f$ belongs to the essential image of $\sf{OPCA}\to\sf{PCA}$. Equivalently, if $\tilde{f}$ belongs to the essential image of $T$. \ruitje
\end{defn}
In other words, $f$ is projective iff there exists a morphism of OPCAs $f_0\colon A\to B$ such that $f\simeq \delta_B f_0$, and in this case, we have $\tilde{f}\simeq Tf_0$. In fact, it suffices that there be a \emph{function} $f_0\colon A\to B$ such that $f\simeq \delta_B f_0$; such an $f_0$ will the automatically be a morphism of OPCAs. At various occasions in the remainder of the paper, we will view morpismsm of OPCAs as projective applicative morpisms.

The following result was obtained in \cite{faberjaap} (Corollary 1.15), using an analysis of the corresponding realizability toposes (to be defined in \cref{sec:ass} below), but it can also be proved directly. It is worth noting that the proof uses the Axiom of Choice.

\begin{thm}\label{thm:criterion_rightadjoint}
An applicative morphism has a right adjoint in $\sf{PCA}$ if and only if it is both projective and c.d.
\end{thm}
\begin{proof}
First, suppose that $f\colon A\am B$ has a right adjoint $g\colon B\am A$. We already know from \cref{prop:cd} that this implies that $f$ is c.d.  For projectivity, suppose that $r\in A$ realizes $\id_A\leq gf$ and $s\in B$ realizes $fg\leq\id_B$. Then for all $a\in A$, we have that $ra\denotes$ and $ra\in gf(a) = \bigcup_{b\in f(a)} g(b)$. By the Axiom of Choice, there exists a function $f_0\colon A\to B$ such that $f_0(a)\in f(a)$ and $ra\in g(f_0(a))$ for all $a\in A$. We claim that $f\simeq \delta_Bf_0$. First of all, we have that $\downset\{f_0(a)\}\subseteq f(a)$, so the identity combinator $\sf{i}$ realizes $\delta_Bf_0\leq f$. The converse inequality is realized by $s':= \lambda^* x.s(tr'x)\in B$, where $r'$ is an element from $f(r)$ and $t\in B$ tracks $f$. Indeed, if $b\in f(a)$, then $tr'b\in f(ra)\subseteq \bigcup_{a'\in g(f_0(a))} f(a') = fg(f_0(a))$. So we see that $s'b\preceq s(tr'b)$, which is defined and an element of $\id_B(f_0(a)) = \downset\{f_0(a)\}$, as desired.

For the converse, let $f\colon A\to B$ be a c.d.\@ morphism of OPCAs; we need to show that $f'=\delta_Bf\colon A\am B$ has a right adjoint $g\colon B\am A$. Let $m\in B$ satisfy \cref{eq:cdm} from \cref{prop:cdm} for $f$. We define $g\colon B\am A$ by:
\[
g(b) = \downset\{a\in A\mid m\cdot f(a)\leq b\}.
\]
First, let us show that $g$ is indeed an applicative morphism. Because $m$ also satisfies \cref{eq:cd} from \cref{defn:cd} for $f$, we know that $g(b)$ is nonempty for every $b\in B$. Moreover, $g$ clearly preserves the order on the nose. In order to construct a tracker, let
\[
s = \lambda^* x.m(u(t\cdot f(\sf{p}_0)\cdot x))(m(u(t\cdot f(\sf{p}_1)\cdot x))) \in B,
\]
where $t$ tracks $f$ and $f$ preserves the order up to $u$. Find $r\in A$ such that $m\cdot f(ra)\preceq s\cdot f(a)$, and define $q = \lambda^* xy.r(\sf{p}xy)\in A$. We claim that $q$ tracks $g$. We need to show that, if $bb'\denotes$, then
\[
q\cdot g(b)\cdot g(b') = \downset\{qaa'\mid m\cdot f(a)\leq b\mbox{ and }m\cdot f(a')\leq b\}
\]
is a subset of $g(bb')$. So suppose that $m\cdot f(a)\leq b$ and $m\cdot f(a')\leq b$. Then $qaa'\preceq r(\sf{p}aa')$ and:
\begin{align*}
m\cdot f(r(\sf{p}aa')) &\preceq s\cdot f(\sf{p}aa')\\
&\preceq m(u(t\cdot f(\sf{p}_0)\cdot f(\sf{p}aa')))(m(u(t\cdot f(\sf{p}_1)\cdot f(\sf{p}aa'))))\\
&\preceq m(u\cdot f(\sf{p}_0(\sf{p}aa')))(m(u\cdot f(\sf{p}_1(\sf{p}aa'))))\\
&\preceq m\cdot f(a)(m\cdot f(a'))\\
&\preceq bb',
\end{align*}
so $qaa'\in g(bb')$, as desired.

In order to establish the adjunction $f'\dashv g$, we first note that 
\[gf'(a) = \bigcup_{b\leq f(a)} g(b) = \downset\{a'\in A\mid m\cdot f(a')\leq f(a)\}.\]
According to \cref{eq:cdm}, there exists an $r\in A$ such that $m\cdot f(ra) \preceq \sf{i}\cdot f(a) \leq f(a)$ for all $a\in A$. This immediately implies that $ra\in gf'(a)$ for all $a\in A$, so $r$ realizes $\id_A\leq gf'$. Conversely, we have
\[
f'(g(b)) = \bigcup_{a\in g(b)} \downset\{f(a)\} = \downset\{f(a)\mid m\cdot f(a)\leq b\},
\]
so it is immediate that $m\in B$ realizes $f'g\leq\id_B$.
\end{proof}

We observe that, as an immediate corollary of this, any two OPCAs that are equivalent in $\sf{PCA}$ are already equivalent in $\sf{OPCA}$. This means that we can speak unambiguously about the equivalence of OPCAs.


\section{Products and coproducts in $\mathsf{PCA}$}\label{sec:prod_PCA}

In this section, we investigate to which extent the results from \cref{sec:prod_OPCA} carry over to the category $\sf{PCA}$. For pseudocoproducts, this is quite easy.

\begin{cor}
The pseudofunctor $\sf{OPCA}\to\sf{PCA}$ preserves finite pseudocoproducts. In particular, $\sf{PCA}$ has all finite pseudocoproducts.
\end{cor}
\begin{proof}
For every OPCA $A$, we have $\sf{PCA}(\mathbf{1},A)\simeq \sf{OPCA}(\mathbf{1},TA)$, which we know to be equivalent to the one-element preorder. Similarly, if $A_0$, $A_1$ and $B$ are OPCAs, then
\begin{align*}
\sf{PCA}(A_0\times A_1,B)&\simeq \sf{OPCA}(A_0\times A_1,TB)\\
&\simeq \sf{OPCA}(A_0,TB)\times \sf{OPCA}(A_1,TB)\\
&\simeq \sf{PCA}(A_0,B)\times \sf{PCA}(A_1,B),
\end{align*}
finishing the proof.
\end{proof}
Explicitly, if $f_0\colon A_0\am B$ and $f_1\colon A_1\am B$ are applicative morphisms, then their amalgamation $[f_0,f_1]\colon A_0\times A_1\am B$ is given by:
\[
[f_0,f_1](a_0,a_1) = \downset\{\sf{p}b_0b_1\mid b_0\in f_0(a_0)\mbox{ and }b_1\in f_1(a_1)\}.
\]
By \cref{prop:cd}(ii) (or rather, its counterpart for $\sf{PCA}$), we immediately have the following corollary.
\begin{cor}
If $f_0\colon A_0\am B$ and $f_1\colon A_1\am B$ are applicative morphisms and $f_0$ is c.d., then $[f_0,f_1]\colon A_0\times A_1\am B$ is also c.d.
\end{cor}
Since $T\mathbf{1}\simeq\mathbf{1}$, we have that $\mathbf{1}$ is not only pseudoinitial in $\sf{PCA}$, but also pseudoterminal. Therefore, we also define zero morphisms in $\sf{PCA}$, by saying that $f\colon A\am B$ is a zero morphism iff it factors (in $\sf{PCA}$) through $\mathbf{1}$. This is in fact equivalent to $f\colon A\to TB$ being a zero morphism in $\sf{OPCA}$, which is equivalent to $\bigcap_{a\in A}f(a)\neq\emptyset$. The proof of the following proposition is now completely analogous to the proof \cref{prop:OPCA_disjoint}, and is therefore omitted.
\begin{prop}
Pseudocoproducts in $\sf{PCA}$ are disjoint.
\end{prop}

If we want to show that $A_0\times A_1$ is also the pseudoproduct of $A_0$ and $A_1$ in $\sf{PCA}$, then we should show that $T(A_0\times A_1)\simeq TA_0\times TA_1$. However, it turns out that this is \emph{not} true in general, and that $\sf{PCA}$ does not have finite pseudoproducts. On the other hand, $A_0\times A_1$ is still a product of $A_0$ and $A_1$ in $\sf{PCA}$ in a weak sense. Explicitly, if $f_0\colon B\am A_0$ and $f_1\colon B\am A_1$, then there exists a \emph{maximal} mediating arrow $f\colon B\am A_0\times A_1$. Using the theory developed in \cref{sec:prod_OPCA}, we can tie things together quite nicely.

Because $T$ is a pseudofunctor, we have arrows $T\pi_0\dashv T\kappa_0\colon TA_0\to T(A_0\times TA_1)$ and $T\pi_1\dashv T\kappa_1\colon TA_1\to T(A_0\times TA_1)$ of $\sf{OPCA}_\adj$. By \cref{prop:coprod_OPCA_adj}, there exists a mediating arrow $h^*\dashv h_*\colon TA_0\times TA_1\to T(A_0\times TA_1)$. Explicitly, we have $h_*(\alpha_0,\alpha_1) = \alpha_0\times\alpha_1$ for $\alpha_i\in TA_i$, whereas
\begin{align*}
h^*(\alpha) &= (T\pi_0(\alpha),T\pi_1(\alpha))\\
&= (\{a_0\in A_0\mid \exists a_1\in A_1\s ((a_0,a_1)\in \alpha)\},\{a_1\in A_1\mid \exists a_0\in A_0\s ((a_0,a_1)\in \alpha)\})
\end{align*}
for $\alpha\in T(A_0\times A_1)$. One easily computes that $h^* h_*$ is in fact isomorphic to $\id_{TA_0\times TA_1}$. (This also follows from the fact that $T\pi_i\circ T\kappa_i\simeq \id_{TA_i}$, whereas $T\pi_j\circ T\kappa_i$ is a zero morphism for $i\neq j$.) Now we see that
\begin{align*}
\sf{PCA}(B,A_0)\times \sf{PCA}(B,A_1) &\simeq \sf{OPCA}(B,TA_0)\times \sf{OPCA}(B,TA_1)\\
&\simeq \sf{OPCA}(B,TA_0\times TA_1)\\
&\leftrightarrows \sf{OPCA}(B,T(A_0\times A_1))\\
&\simeq \sf{PCA}(B,A_0\times A_1),
\end{align*}
where
\[
\begin{tikzcd}
{\sf{OPCA}(B,TA_0\times TA_1)} \arrow[rr, "h_*\circ -"', shift right=2.5] \arrow[rr, phantom, "\perp" description] &  & {\sf{OPCA}(B,T(A_0\times TA_1))} \arrow[ll, " h^*\circ -"', shift right=2]
\end{tikzcd}
\]
is an adjunction whose counit is an isomorphism. In particular, if $f_0\colon B\am A_0$ and $f_1\colon B\am A_1$ are applicative morphisms, then
\[
\begin{tikzcd}[column sep=large]
B \arrow[r, "{\langle f_0,f_1\rangle}"] & TA_0\times TA_1 \arrow[r, "h_*"] & T(A_0\times A_1)
\end{tikzcd}
\]
is the maximal mediating applicative morphism $B\am A_0\times A_1$. Conversely, $g\colon B\am A_0\times A_1$ is such a maximal mediating morphism iff $g\colon B\to T(A_0\times A_1)$ factors through $h_*$; or equivalently, $h_* h^* g\simeq g$. Observe that this includes all \emph{projective} $g\colon B\am A_0\times A_1$. Indeed if $g\simeq \delta_{A_0\times A_1}\circ g_0$ with $g_0\colon B\to A_0\times A_1$, then we also have $g\simeq \delta_{A_0\times A_1} \circ g_0 \simeq h_\ast \circ(\delta_{A_0}\times \delta_{A_1})\circ g_0$.

The above shows that pseudoproducts exist in in $\sf{PCA}$ in a weak sense. Now let us turn to the existence of actual pseudoproducts in $\sf{PCA}$. Obviously, if $A_0$ (resp.\@ $A_1$) is trivial, then the pseudoproduct of $A_0$ and $A_1$ exists in $\sf{PCA}$, and it is equivalent to $A_1$ (resp.\@ $A_0$). Using the morphism $h^\ast$ above, we can show that this is the \emph{only} situation in which $A_0$ and $A_1$ have a product in $\sf{PCA}$.
\begin{thm}\label{thm:no_prod}
If $A_0$ and $A_1$ are OPCAs that have a pseudoproduct in $\sf{PCA}$, then at least one of $A_0$ and $A_1$ is trivial.
\end{thm}

\begin{proof}
The proof is divided into two parts.
\begin{enumerate}
\item	First, we show that $h^\ast\colon T(A_0\times A_1)\to TA_0\times TA_1$ has a left adjoint, and is therefore discrete.
\item	Second, we show that $h^\ast$ cannot be discrete if $A_0$ and $A_1$ are both nontrivial.
\end{enumerate}
For the first part, denote the pseudoproduct projections $TA_0\times TA_1\to TA_i$ by $\rho_i$; then $h^\ast$ is the essentially unique morphism such that
\[
\begin{tikzcd}
T(A_0\times A_1) \arrow[rr, "h^*"] \arrow[rd, "T\pi_i"'] &      & TA_0\times TA_1 \arrow[ld, "\rho_i"] \\
                                                            & TA_i &                                     
\end{tikzcd}
\]
commutes up to isomorphism, for $i=0,1$.

Suppose that $C$ is a pseudoproduct of $A_0$ and $A_1$ in $\sf{PCA}$, with projections $\sigma_i\colon C\am A_i$. Then $\sigma_0$ and $\sigma_1$ induce a maximal mediating arrow $f\colon C\am A_0\times A_1$. On the other hand, $\pi_0$ and $\pi_1$, seen as projective applicative morphisms, induce a unique mediating map $g\colon A_0\times A_1\am C$. So for $i=0,1$ we get a diagram in $\sf{PCA}$:
\begin{equation}\label{diag:product}
\begin{tikzcd}
A_0\times A_1 \arrow[rr, -multimap, "g"', shift right] \arrow[rd, "\pi_i"'] &     & C \arrow[ll, -multimap, "f"', shift right] \arrow[ld, -multimap, "\sigma_i"] \\
                                                                 & A_i &                                                       
\end{tikzcd}
\end{equation}
where the triangles commute up to isomorphism. Since $C$ is a pseudoproduct, we have $gf\simeq\id_C$. Moreover, we have $\pi_i fg\simeq \sigma_i g\simeq \pi_i \simeq \pi_i\circ \id_{A_0\times A_1}$ for $i=0,1$, and since $\id_{A_0\times A_1}$ is certainly projective, this yields $fg\leq \id_{A_0\times A_1}$. We can conclude that $f\dashv g$.

For every OPCA $B$, we have natural equivalences
\begin{align*}
\sf{OPCA}(B,TC) &\simeq \sf{PCA}(B,C)\\
&\simeq \sf{PCA}(B,A_0)\times\sf{PCA}(B,A_1)\\
&\simeq \sf{OPCA}(B,TA_0)\times\sf{PCA}(B,TA_1),
\end{align*}
so $TA_0\stackrel{\tilde{\sigma}_0}{\longleftarrow} TC\stackrel{\tilde{\sigma}_1}{\longrightarrow} TA_1$ is a product diagram in $\sf{OPCA}$. This means there exists an equivalence $\iota\colon TC\to TA_0\times TA_1$ such that the diagram
\[
\begin{tikzcd}
TC \arrow[rr, "\iota"] \arrow[rd, "\tilde{\sigma}_i"'] &      & TA_0\times TA_i \arrow[ld, "\rho_i"] \\
                                                         & TA_i &                                  
\end{tikzcd}
\]
commutes up to isomorphism for $i=0,1$. Taking the image of the diagram \cref{diag:product} under the equivalence between $\sf{PCA}$ and free $T$-algebras, we get the diagram
\[\begin{tikzcd}
T(A_0\times A_1) \arrow[rd, "T\pi_i"'] \arrow[r, "\tilde{g}"', shift right] & TC \arrow[r, "\iota"] \arrow[d, "\tilde{\sigma}_i"] \arrow[l, shift right, "\tilde{f}"'] & TA_0\times TA_1 \arrow[ld, "\rho_i"] \\
                                                                            & TA_i                                                                        &                                     
\end{tikzcd}\]
in $\sf{OPCA}$ for $i=0,1$, where all triangles commute up to isomorpism. In particular, $\rho_i\iota \tilde{g}\simeq \tilde{\sigma}_i \tilde{g}\simeq T\pi_i$, so $\iota\tilde{g}$ must be isomorphic to $h^*$. Since $f\dashv g$, we also have $\tilde{f}\dashv \tilde{g}$, hence also $\tilde{f}\iota^{-1}\dashv \iota \tilde{g}\simeq h^*$. We conclude that $h^*$ has a left adjoint, so by \cref{prop:cd}, $h^*$ is discrete.

For the second part, suppose that $A_0$ and $A_1$ are both nontrivial, and that $h^\ast$ is discrete. Consider the set
\[
X \subseteq \{\alpha\in T(A_0\times A_1)\mid h^\ast(\alpha) = (A_0,A_1)\}.
\]
We claim that $\bigcap X$ is empty. Let $(a_0,a_1)\in A_0\times A_1$ be arbitrary, and consider the downset
\[
\alpha = \{(b_0,b_1)\in A_0\times A_1\mid a_0\nleq b_0\mbox{ or }a_1\nleq b_1\}
\]
of $A_0\times A_1$. Since $a_0$ is, by assumption, not the least element of $A_0$, there exists a $b_0\in A_0$ such that $a_0\nleq b_0$. This implies that $\{b_0\}\times A_1\subseteq \alpha$, so $\alpha$ is nonempty and satisfies $T\pi_1(\alpha) = A_1$. Similarly, we show that $T\pi_0(\alpha) = A_0$, so $\alpha\in X$. On the other hand, we clearly do \emph{not} have $(a_0,a_1)\in \alpha$, so $(a_0, a_1)\not\in \bigcap X$. Since this holds for all $(a_0,a_1)\in A_0\times A_1$, we can conclude that $\bigcap X=\emptyset$.

But $h^\ast(X) = \{(A_0, A_1)\}$ obviously has a lower bound in $TA_0\times TA_1$, so since $h^\ast$ is discrete, $X$ should have a lower bound in $T(A_0\times A_1)$. However, this is impossible given that $\bigcap X$ is empty, so we have reached a contradiction.
\end{proof}

We close this section by investigating, in analogy with $\sf{OPCA}_\adj$, the category $\sf{PCA}_\adj$.

\begin{defn}\label{defn:PCAadj}
The preorder-enriched category $\sf{PCA}_\adj$ is defined as follows.
\begin{itemize}
\item	Its objects are OPCAs.
\item	An arrow $f\colon A\to B$ is a pair of applicative morphisms $f^*\colon B\am A$ and $f_*\colon A\am B$ with $f^*\dashv f_*$.
\item	If $f,g\colon A\to B$, then we say that $f\leq g$ if $f^*\leq g^*$; equivalently, if $g_*\leq f_*$. \ruitje
\end{itemize}
\end{defn}

From \cref{thm:criterion_rightadjoint}, we know that $\sf{PCA}_\adj$ is actually equivalent to $\sf{OPCA}_\text{cd}^\text{op}$, where $\sf{OPCA}_\text{cd}$ denotes the wide subcategory of $\sf{OPCA}$ on the c.d.\@ morphisms, and $(\cdot)^\text{op}$ indicates a reversal of the \emph{1-cells}. The following result is now immediate.
\begin{cor}\label{cor:coprod_PCA_adj}
The category $\sf{PCA}_\adj$ has finite pseudocoproducts. Moreover, the pseudoinitial object is strict, and pseudocoproducts are disjoint.
\end{cor}
\begin{proof}
It suffices to prove the dual statements in $\sf{OPCA}_\text{cd}$. By \cref{prop:cd_prod_OPCA}, $\sf{OPCA}_\text{cd}$ has \emph{finite} pseudoproducts. Moreover, by \cref{lem:trivial}, the terminal object is strict in $\sf{OPCA}_\text{cd}$. The final statement is \cref{prop:dual_disjoint}.
\end{proof}


\section{The realizability topos}\label{sec:ass}

In this final section, we briefly investigate what we can say about coproducts of the \emph{realizability toposes} associated to OPCAs; in particular, to which extent realizability toposes are closed under coproducts. First, let us give the appropriate definitions.

\begin{defn}
Let $A$ be an OPCA.
\begin{itemize}
\item[(i)]	An \emph{assembly} over $A$ is a pair $X = (|X|,E_X)$, where $|X|$ is a set, and $E_X$ is a function $|X|\to TA$.
\item[(ii)]	A \emph{morphism of assemblies} $X\to Y$ is a function $f\colon X\to Y$ for which there exists an $r\in A$ (called a \emph{tracker} of $f$) such that $r\cdot E_X(x)\subseteq E_Y(f(x))$ for all $x\in |X|$. \ruitje
\end{itemize}
\end{defn}
Assemblies and morphisms between them form a quasitopos $\asm(A)$. Moreover, there is an obvious forgetful funtor $\Gamma_A\colon \asm(A)\to\set$ sending $X$ to $|X|$, and there is a functor $\nabla_A\colon\set\to\asm(A)$, sending a set $Y$ to the assembly $(Y,y\mapsto A)$. These functors are both regular, and they satisfy $\Gamma_A\dashv\nabla_A$ with $\Gamma_A\nabla_A\cong\id_\set$.

The ex/reg completion of $\asm(A)$ turns out to be a topos, which is called the \emph{realizability topos} of $A$ and denoted by $\rt(A)$. Since there is an inclusion $\asm(A)\mono\rt(A)$, we can also view $\nabla_A$ as a functor $\set\to \rt(A)$. Moreover, since $\Gamma_A$ is regular and $\set$ is exact, $\Gamma_A$ may be lifted to a functor $\rt(A)\to \set$, which we denote by $\hat{\Gamma}_A$. This yields an adjunction
\[
\begin{tikzcd}
\set \arrow[r, "\nabla_A"', shift right] & \rt(A) \arrow[l, "\hat{\Gamma}_A"', shift right]
\end{tikzcd}
\]
where $\hat{\Gamma}_A\nabla_A\cong\id_\set$ and $\hat{\Gamma}_A$ preserves finite limits. This means that $\set$ is a subtopos of $\rt(A)$, and in fact, this is precisely the inclusion of double negation sheaves. The $\neg\neg$-separated objects are precisely those objects that are isomorphic to an assembly.

The following result was first obtained by J.\@ Longley for the unordered case (\cite{longleyphd}, Theorem 2.3.4), and generalized to OPCAs in \cite{hofstrajaap}. We denote by $\reg$ the 2-category of regular categories, regular functors, and natural transformations. Moreover, $\reg/\set$ will denote the pseudoslice of $\reg$ over $\set$, i.e., its objects are regular functors with codomain $\set$, its 1-cells are triangles that commute up to specified isomorphism, and its 2-cells are natural transformations that are compatible with these specified isomorphisms.

\begin{thm}\label{thm:longley}
The assignment $A\mapsto (\Gamma_A\colon \asm(A)\to\set)$ may be extended to a local equivalence $\sf{PCA}\to \reg/\set$.
\end{thm}

Let $A_0$ and $A_1$ be OPCAs. The pseudocoproduct of $\rt(A_0)$ and $\rt(A_1)$, in the 2-category of toposes and geometric morphisms, is the product category $\rt(A_0)\times \rt(A_1)$. In this topos, the logic may be computed componentwise, which implies that its subtopos of double negation sheaves is equivalent to $\set^2$, rather than $\set$. This immediately tells us that $\rt(A_0)\times \rt(A_1)$ is never equivalent to a realizability topos. It should be mentioned, however, that $(A_0,A_1)$ is an OPCA internal to the topos $\set^2$, and that constructing $\rt(A_0,A_1)$ over the base $\set^2$ rather than $\set$ does yield $\rt(A_0)\times \rt(A_1)$. See also the treatment in \cite{Z19}.

If we want to keep working over the base $\set$, on the other hand, then it makes more sense to take the pseudocoproduct over $\set$. That is, we consider the pseudopushout square
\[\begin{tikzcd}
\set \arrow[d, hook] \arrow[r, hook] \arrow[rd,hook] & \rt(A_0) \arrow[d, hook] \\
\rt(A_1) \arrow[r, hook]               & \mathcal{E}           
\end{tikzcd}\]
which always exists according to Proposition 4.26 from \cite{topos}. This proposition also tells us that the inverse image part of this diagram:
\[\begin{tikzcd}
\mathcal{E} \arrow[r] \arrow[d] \arrow[rd] & \rt(A_0) \arrow[d, "\hat{\Gamma}_{A_0}"] \\
\rt(A_1) \arrow[r, "\hat{\Gamma}_{A_1}"']   & \set                        
\end{tikzcd}\]
is a pseudo\emph{pullback} of categories. Because all displayed functors are regular, this is also a pseudopullback in $\reg$, as is not difficult to show. This means that the inverse image part $\mathcal{E}\to \set$ is the pseudoproduct of $\hat{\Gamma}_{A_0}$ and $\hat{\Gamma}_{A_1}$ in $\reg/\set$.

We finish the paper by determining when $\mathcal{E}$ above is itself a realizability topos. If $A_0$ is trivial, then the inclusion $\set\to \rt(A_0)$ is an equivalence, so in that case, we will have $\mathcal{E}\simeq \rt(A_1)$. Similarly, if $A_1$ is trivial, then $\mathcal{E}$ will be equivalent to the realizability topos over $A_0$. It turns out that these are the only cases in which $\mathcal{E}$ is a realizability topos.
\begin{prop}
Let $A_0$ and $A_1$ be OPCAs such that the pseudocoproduct of $\rt(A_0)$ and $\rt(A_1)$ over $\set$ is again a realizability topos. Then at least one of $A_0$ and $A_1$ is trivial.
\end{prop}
\begin{proof}
Suppose that the $\mathcal{E}$ constructed above is equivalent to $\rt(C)$ for some OPCA $C$. By Corollary 1.4 from \cite{johnstone}, there exists (up to isomorphism) at most one geometric morphism $\set\to\rt(C)$. In particular, $\set\mono\mathcal{E}\simeq \rt(C)$ is isomorphic to the inclusion of double negation sheaves. This means that the inverse image part $\rt(C)\to\set$ is isomorphic to $\hat{\Gamma}_C$, so we have a pseudopullback
\[\begin{tikzcd}
\rt(C) \arrow[d, "p_1"'] \arrow[r, "p_0"] \arrow[rd, "\hat{\Gamma}_C"] & \rt(A_0) \arrow[d, "\hat{\Gamma}_{A_0}"] \\
\rt(A_1) \arrow[r, "\hat{\Gamma}_{A_1}"']                              & \set                                    
\end{tikzcd}\]
of categories, where $p_i$ denotes the inverse image of $\rt(A_i)\mono \mathcal{E} \simeq \rt(C)$. By \cite{johnstone}, Lemma 2.4, such an inverse image functor always commutes with the constant object functors, i.e., we have $p_i\nabla_C\simeq \nabla_{A_i}$ for $i=0,1$.

An object $X$ of $\rt(C)$ is isomorphic to an assembly if and only if $X\to \nabla_C\hat{\Gamma}_C X$ is a monomorphism. By the pseudopullback diagram above, this is the case iff and $p_i X\to p_i\nabla_C\hat{\Gamma}_C X$ is mono for $i=0,1$. Since $p_i\nabla_C\hat{\Gamma}_C X \cong \nabla_{A_i}\hat{\Gamma}_{A_i}p_i X$, this is equivalent to saying that $p_iX$ is isomorphic an assembly, for $i=0,1$. So we also have a pseudopullback
\[
\begin{tikzcd}
\asm(C) \arrow[d] \arrow[r] \arrow[rd, "\Gamma_C"] & \asm(A_0) \arrow[d, "\Gamma_{A_0}"] \\
\asm(A_1) \arrow[r, "\Gamma_{A_1}"']               & \set                               
\end{tikzcd}
\]
of categories. But again, all the displayed functors are regular, so this is also a pseudopullback in $\reg$, meaning that $\Gamma_C$ is a pseudoproduct of $\Gamma_{A_0}$ and $\Gamma_{A_1}$ in $\reg/\set$.

This, together with \cref{thm:longley}, implies that for any OPCA $B$, we have natural equivalences:
\begin{align*}
\sf{PCA}(B,C) &\simeq (\reg/\set)(\Gamma_B,\Gamma_C)\\
&\simeq (\reg/\set)(\Gamma_B,\Gamma_{A_0})\times (\reg/\set)(\Gamma_B,\Gamma_{A_1})\\
&\simeq \sf{PCA}(B,A_0)\times\sf{PCA}(B,A_1),
\end{align*}
so $C$ is a pseudoproduct of $A_0$ and $A_1$ in $\sf{PCA}$. Applying \cref{thm:no_prod} finishes the proof.
\end{proof}

Even though the pushout $\mathcal{E}$ constructed above is not a realizability topos, we can ask how it is from being a realizability topos. The adjunctions $\pi_i\dashv \kappa_i$ between $A_i$ and $A_0\times A_1$ give rise to geometric inclusions $\rt(A_i)\mono \rt(A_0\times A_1)$. The pushout diagram above then also yields a geometric inclusion $\mathcal{E}\mono\rt(A_0\times A_1)$, so $\mathcal{E}$ is a subtopos of a realizability topos. We can wonder from which local operator on $\rt(A_0\times A_1)$ this subtopos $\mathcal{E}$ arises. Local operators on a realizability topos $\rt(B)$ arise from \emph{functions} $J\colon DB\to DB$ where $DB$ stands for the set of \emph{all} downsets of $B$ (including $\emptyset$), and $J$ should satisfy certain requirements analogous to the axioms for a local operator. For details, we refer to \cite{oostenlee}. In this particular case, the subtopos $\mathcal{E}$ arises from $J\colon D(A_0\times A_1)\to D(A_0\times A_1)$ defined by
\[
J(\alpha) = \{a_0\in A_0\mid \exists a_1\in A_1\s ((a_0,a_1)\in\alpha)\}\times \{a_1\in A_1\mid \exists a_0\in A_0\s ((a_0,a_1)\in\alpha)\},
\]
i.e., $J(\alpha)$ is the smallest `rectangular' subset of $A_0\times A_1$ containing $\alpha$. We can also describe this map by saying that $J(\alpha) = h_\ast(h^\ast(\alpha))$ for $\alpha \in T(A_0\times A_1)$ (with $h^\ast\dashv h_\ast$ as in the previous section), and $J(\emptyset) = \emptyset$.


\begin{thebibliography}{HvO03}

\bibitem[FvO14]{faberjaap}
E.~Faber and J.~van Oosten.
\newblock More on geometric morphisms between realizability toposes.
\newblock {\em Theory and Applications of Categories}, 29(30):874--95, 2014.

\bibitem[Hof06]{allrealisrel}
P.~Hofstra.
\newblock All realizability is relative.
\newblock {\em {Math.\@ Proc.\@ Camb.\@ Phil.\@ Soc.}}, 141(2):239--64, 2006.

\bibitem[HvO03]{hofstrajaap}
P.~Hofstra and J.~van Oosten.
\newblock Ordered partial combinatory algebras.
\newblock {\em Math. Proc. Camb. Phil. Soc.}, 134(3):445--463, 2003.

\bibitem[Joh77]{topos}
P.~T. Johnstone.
\newblock {\em Topos Theory}.
\newblock {Academic Press}, 1977.
\newblock Paperback edition: Dover reprint 2014.

\bibitem[Joh13]{johnstone}
P.~T. Johnstone.
\newblock Geometric morphisms of realizability toposes.
\newblock {\em Theory and Applications of Categories}, 28(9):241--249, 2013.

\bibitem[Lon94]{longleyphd}
J.~Longley.
\newblock {\em Realizability Toposes and Language Semantics}.
\newblock PhD thesis, University of Edinburgh, 1994.

\bibitem[LvO13]{oostenlee}
S.~Lee and J.~van Oosten.
\newblock Basis subtoposes of the effective topos.
\newblock {\em {Annals of Pure and Applied Logic}}, 164(9):335--47, 2013.

\bibitem[vO08]{jaap}
J.~van Oosten.
\newblock {\em Realizability: An Introduction to its Categorical Side}, volume
  152 of {\em Studies in Logic and the Foundations of Mathematics}.
\newblock Elsevier, 2008.

\bibitem[Zoe19]{Z19}
J.~Zoethout.
\newblock Internal partial combinatory algebras and their slices.
\newblock {\em ArXiv e-prints}, 2019.
\newblock \url{https://arxiv.org/abs/1910.09816v1}.

\end{thebibliography}
\end{document}